\documentclass[11pt]{amsart}
\pdfoutput=1
\usepackage{graphicx}
\usepackage{amssymb, amsmath, amscd}
\usepackage{float}
\usepackage[mathscr]{euscript}
\usepackage{array}
\usepackage{chngcntr}
\usepackage[bookmarks=false]{hyperref}

\newtheorem{theorem}{Theorem}[section]
\newtheorem{corollary}[theorem]{Corollary}
\newtheorem{lemma}[theorem]{Lemma}

\newtheorem{proposition}[theorem]{Proposition}

\newtheorem{definition-proposition}[theorem]{Definition-Proposition}
\newtheorem{lemma-notation}[theorem]{Lemma-Notation}

\theoremstyle{definition}
\newtheorem{definition}[theorem]{Definition}

\newtheorem{example}[theorem]{Example}
\newtheorem{remark}[theorem]{Remark}

\newtheorem{convention}[theorem]{Convention}

\newcommand{\Z}{\mathbb{Z}}

\newcommand{\R}{\mathbb{R}}
\newcommand{\C}{\mathbb{C}}

\newcommand{\twobytwo}[4]
{
	\begin{pmatrix}
		#1 & #2 \\
		#3 & #4 \\
	\end{pmatrix}
}

\begin{document}
\title[The number of convex tilings of the sphere]{The number of convex tilings of the sphere by triangles, squares, or hexagons}
\author[Philip Engel]{Philip Engel$^\dagger$}
\address[Philip Engel]{Harvard University}
\thanks{$^\dagger$Research partially supported by NSF grant DMS-1502585.}
\email{engel@math.harvard.edu}

\author{Peter Smillie}
\address[Peter Smillie]{Harvard University}
\email{smillie@math.harvard.edu}

\maketitle

\begin{abstract}
A tiling of the sphere by triangles, squares, or hexagons is {\it convex} if every vertex has at most $6$, $4$, or $3$ polygons adjacent to it, respectively. Assigning an appropriate weight to any tiling, our main result is explicit formulas for the weighted number of convex tilings with a given number of tiles. To prove these formulas, we build on work of Thurston, who showed that the convex triangulations correspond to orbits of vectors of positive norm in a Hermitian lattice $\Lambda\subset \mathbb{C}^{1,9}$. First, we extend this result to convex square- and hexagon-tilings. Then, we explicitly compute the relevant lattice $\Lambda$. Next, we integrate the Siegel theta function for $\Lambda$ to produce a modular form whose Fourier coefficients encode the weighted number of tilings. Finally, we determine the formulas using finite-dimensionality of spaces of modular forms.
\end{abstract}

\section{Introduction}

%


A tiling of the sphere by triangles, squares, or hexagons is {\it convex} if every vertex is adjacent to at most six, four, or three polygons, respectively. In this paper, we count convex tilings of $S^2$ up to combinatorial equivalence. Convexity is a very strong restriction; while the total number of tilings grows exponentially \cite{tutte} in the number of tiles, the number of convex tilings grows polynomially. Thurston \cite{thurston} proved that convex triangulations with $2n$ triangles correspond to $U(\Lambda)$-orbits of vectors of norm $2n$ in a Hermitian lattice $\Lambda\subset \mathbb{C}^{1,9}$. It follows that the number of triangulations with less than $2n$ triangles is order $n^{10}$. Similarly, we show that convex square- and hexagon-tilings correspond to lattice points in $\C^{1,5}$ and $\C^{1,3}$ respectively. 


It is natural to weight a tiling by the inverse of the order of the $U(\Lambda)$-stabilizer of the associated vector. The weight is expressed in terms of the tiling itself in Definition \ref{weight}. Counting with weight does not change the order of growth, and allows for simple exact formulas for the number of convex tilings. Following Convention \ref{convention} in the case of hexagon-tilings, our main theorem is:

\begin{theorem}\label{intro} The weighted number of oriented convex tilings of $S^2$ with $n$ tiles is

\begin{centering}

\begin{tabular} { l l }

$\displaystyle\frac{809}{2^{15}3^{13}5^2}\sigma_9(n/2)$ & for triangles, \\[0.5em]

$\displaystyle\frac{1}{2^{13}3^2}(\sigma_5(n)+8\sigma_5(n/2))$ & for squares, and  \\[0.5em]

$\displaystyle\frac{1}{2^33^4}(\sigma_3(n)-9\sigma_3(n/3))$ & for hexagons,\\[0.7em] \end{tabular}

\end{centering} \noindent where $\sigma_m(n)=\sum_{d\mid n} d^m$ when $n$ is an integer, and otherwise $\sigma_m(n)=0$.

\end{theorem}

The key to proving Theorem \ref{intro} is to show that the generating function for the weighted number of tilings is a modular form of a specific weight and level. The modularity follows from our second theorem:

\begin{theorem}\label{modularity} Let $k=6$, $4$, or $3$. Let $\zeta_k$ be a primitive $k$th root of unity. Let $(\Lambda,\star)$ be a Hermitian lattice over $\Z[\zeta_k]$ of signature $(1,s)$ with $s>1$. Suppose $\Lambda=(1+\zeta_k)\Lambda^\vee$, where $\Lambda^\vee$ denotes the Hermitian dual of $\Lambda$. Let $\Gamma$ be a finite index subgroup of $U(\Lambda)$ and let $\Lambda^+\subset \Lambda$ denote the set of positive norm vectors. Then, there is a constant $c_0$ such that $$c_0+ \sum_{v \in \Gamma\backslash\Lambda^+} \frac{1}{|{\rm Stab}(v)|} \,{\rm exp}\left(\frac{2\pi i \tau v\star v}{|1+\zeta_k|^2}\right)$$ is modular form in $\tau$ of weight $1+s$ with respect to the group $\Gamma_1(|1-\zeta_k|^2)$. \end{theorem}

Theorem \ref{intro} follows from Theorem \ref{modularity} because of the finite-dimensionality of spaces of modular forms. The relevant modular forms are uniquely determined by some small Fourier coefficients, which equal the weighted number of tilings for some small numbers of tiles.

In Section 2, we review Thurston's work on flat cone spheres and extend his results to convex square- and hexagon-tilings. We prove in Proposition \ref{stuff} that convex tilings correspond to $\Gamma$-orbits of positive norm vectors in some Hermitian lattice. In Section 3, we explicitly compute the lattices corresponding to convex tilings by each polygon. By Proposition \ref{dual}, these lattices satisfy the assumptions of Theorem \ref{modularity}. In Section 4, we prove Theorem \ref{modularity}. The technique is to integrate the Siegel theta function---a function on $\mathbb{CH}^s\times \mathbb{H}$ satisfying certain transformation properties---over the complex-hyperbolic orbifold $\Gamma\backslash \mathbb{CH}^s$. This integral is a Maass form on $\mathbb{H}$ whose Fourier coefficients encode the weighted counts of orbits of vectors of given norm. Finally, in Section 5, we deduce Theorem \ref{intro} from Theorem \ref{modularity} by determining the weighted number of tilings with one or two tiles.

 \vspace{5pt}

{\bf Acknowledgements:} We would like to thank Simion Filip, Karsten Gimre, and Curtis McMullen for their suggestions and conversation. In addition, we thank the referee for their comments and suggestions.

\section{Flat Cone Spheres and Tilings}

A {\it flat cone sphere} is a sphere with a flat metric away from a finite set of points $\{p_1,\dots,p_n\}$, which near each point is isometric to a convex Euclidean cone. A convex tiling by triangles, squares, or hexagons has the structure of a flat cone sphere by declaring each tile to be regular of a fixed side length. In 1942, Alexandrov \cite{alexandrov} proved that every flat cone sphere is isometric to a unique convex polyhedron in $\R^3$. Near each singularity $p_i$ the angle of the cone is $2\pi-\alpha_i$ for some $\alpha_i\in(0,2 \pi)$. We call $\alpha_i$ the {\it cone angle deficit}. The Gauss-Bonnet theorem implies that $$\sum \alpha_i=4\pi.$$

Thurston \cite{thurston} studied the moduli space $\mathcal{M}_{\alpha_1,\dots,\alpha_n}$ of all flat cone spheres modulo scaling with specified cone angle deficits at $n$ unmarked points. Using local period maps to complex hyperbolic space $$\mathbb{CH}^{n-3}:=\mathbb{P}\{v\in\C^{1,n-3}\,\big{|}\, v^2>0\},$$ Thurston showed that $\mathcal{M}_{\alpha_1,\dots,\alpha_n}$ is a complex hyperbolic orbifold. Such moduli spaces were studied earlier by Deligne and Mostow \cite{dm}, see also Looijenga's survey paper \cite{looijenga}, in the context of hypergeometric functions and Lauricella differentials, i.e. differentials on $\mathbb{P}^1$ of the form $$\eta=(z-z_1)^{-r_1}\dots(z-z_{n-1})^{-r_{n-1}}\,dz$$ where $r_i\in(0,1)$. When $\sum r_i\in(1,2)$, such a differential induces on $\mathbb{P}^1$ the structure of a flat cone sphere by using $\int_0^z \eta$ as a local flat coordinate. 

The moduli space $\mathcal{M}_{\alpha_1,\dots,\alpha_n}$ is generally metrically incomplete because the cone singularities can collide, in which case the cone angle adds. Therefore Thurston considered the metric completion $\overline{\mathcal{M}}_{\alpha_1,\dots,\alpha_n}$, which is stratified by moduli spaces of flat cone spheres in which some collections of singularities have coalesced, cf. Theorem 3.4 of \cite{thurston}. If the completion is still a complex hyperbolic orbifold, a generalization of the Cartan-Hadamard theorem implies that it is a quotient of $\mathbb{CH}^{n-3}$ by the action of a discrete subgroup $\mathbb{P}\Gamma \subset \mathbb{P}U(1,n-3)$.


The completion $\overline{\mathcal{M}}_{\alpha_1,\dots,\alpha_n}$ is always an orbifold when $n=3$, since the moduli space is a point. For $n=4$, there are infinitely many choices of $\{\alpha_i\}$ for which $\overline{\mathcal{M}}_{\alpha_1,\alpha_2,\alpha_3,\alpha_4}$ is a complex hyperbolic orbifold. Indeed, every quotient by a triangle group $\mathbb{P}\Gamma\subset \mathbb{P}U(1,1)$ acting on $\mathbb{CH}^1=\mathbb{H}^2$ appears as a moduli space of tetrahedra. For $n=5$, Picard \cite{picard} and his student LeVavasseur found the cone angles for which $\overline{\mathcal{M}}_{\alpha_1,\dots,\alpha_5}$ is an orbifold, in the guise of determining when the monodromy group of the hypergeometric function $$F(x,y)=\int_1^\infty z^{-r_1}(z-1)^{-r_2}(z-x)^{-r_3}(z-y)^{-r_4}\,dz$$ is discrete. For $n\geq 5$, Mostow \cite{mostow} and later Thurston enumerated all $94$ values of the $\alpha_i$ for which the moduli space is an orbifold. 

The highest-dimensional completed moduli space of flat cone spheres which is an orbifold corresponds to $n=12$ and $\alpha_i=\pi/3$ for all $i$. Every convex triangulation defines a point in the completion by declaring each triangle metrically equilateral. See \cite{schwartz} for a more detailed treatment of this case. There is a stratification $$\overline{\mathcal{M}}_{\pi/3,\dots,\pi/3}=\!\!\!\!\!\coprod_{\substack{\sum \mu_i=12 \\ \mu_i\in\{1,2,3,4,5\}}}\!\!\!\!\!\mathcal{M}_{\mu_1\pi/3,\,\dots,\,\mu_n\pi/3}.$$ Of the 78 completed moduli spaces of flat cone spheres with $n\geq 5$ which are orbifold quotients of $\mathbb{CH}^{n-3}$ by arithmetic groups, $34$ are strata in the above example. The total number of strata is $45$.

Let $\mathbb{P}\Gamma$ denote the monodromy group of $\overline{\mathcal{M}}_{\pi/3,\dots,\pi/3}$ so that $$\mathbb{P}\Gamma\backslash\C\mathbb{H}^9=\overline{\mathcal{M}}_{\pi/3,\dots,\pi/3}.$$ Thurston observed a special feature of this example. The group $\mathbb{P}\Gamma$ can be lifted to a subgroup $\Gamma\subset U(1,9)$ which preserves a lattice $\Lambda\subset \C^{1,9}$ of geometric significance: The flat cone spheres which admit a triangulation into metrically equilateral triangles are those lying in the projectivization of $\Lambda$. In fact, Thurston showed:

\begin{theorem}[\cite{thurston}, Theorem 0.1] There is a bijective correspondence $$\{\textrm{Oriented convex triangulations of $S^2$}\}\longleftrightarrow \Gamma\backslash \Lambda^+$$ where $\Lambda^+$ is the set of positive norm vectors in $\Lambda$. The number of triangles is the norm of the associated vector. \end{theorem}

We generalize this correspondence to square- and hexagon-tilings in Proposition \ref{stuff}, and construct the lattices explicitly in each case. We first set a convention regarding hexagon-tilings:






\begin{convention}\label{convention} For the remainder of the paper, a {\it hexagon-tiling} denotes a tiling by hexagons such that (1) the vertices of the tiling are bicolored black and white with adjacent vertices assigned different colors, and (2) all vertices of non-zero curvature are black. \end{convention}

\begin{remark} \label{hex and tri} Every hexagon-tiling yields a tiling by triangles by connecting the three black vertices within each hexagon. This triangulation has cone angle deficits which are even multiples of $\pi/3$. Conversely, suppose we have a triangulation with cone angle deficits which are even multiples of $\pi/3$. Because we are tiling the sphere, this condition on cone angle deficits guarantees the existence of a pair of bicolorings of the triangles so that no two adjacent triangles have the same color. For each bicoloring, connect the vertices of every white triangle to its center to produce a hexagon-tiling whose white vertices are the centers of the white triangles. \end{remark}

%
%
%

Having set this convention about hexagon-tilings, we may now proceed in a unified manner for triangles, squares, and hexagons. 

\begin{definition} A {\it Hermitian lattice $\Lambda$ over $\Z[\zeta_k]$} is a finitely-generated, free $\Z[\zeta_k]$-module with a Hermitian pairing valued in $\Z[\zeta_k]$. \end{definition}

\begin{proposition}\label{stuff} Let $k=6$, $4$, or $3$. There is a Hermitian lattice $\Lambda$ of signature $(1,2k-3)$ defined over $\Z[\zeta_k]$, and a group of Hermitian isometries $\Gamma\subset U(\Lambda)$ so that the orbits $\Gamma\backslash \Lambda^+$ of positive norm vectors are in bijective correspondence with oriented convex $\frac{2k}{k-2}$-gon tilings of the sphere.\end{proposition}

\begin{proof} We follow Section 7 of \cite{thurston}. Define a rigidified moduli space $\overline{\mathcal{M}}_{\frac{2\pi}{k}, \dots,\frac{2\pi}{k}}^{rig}$ to be the space of pairs $(M,\Xi)$ where $M\in \overline{\mathcal{M}}_{\frac{2\pi}{k}, \dots,\frac{2\pi}{k}}$ and $\Xi$ is a flat sublattice of the tangent bundle of $M\backslash\{p_1,\dots,p_n\}$ which is locally isometric to $\Z[\zeta_k]\subset \C$ with the standard metric. Let $\mathcal{M}_{\frac{2\pi}{k}, \dots,\frac{2\pi}{k}}^{rig}$ be open stratum where no singularities have collided. When $k=3$, we include with the data of $\Xi$ one of the two choices of hexagonal tiling of the tangent space whose black vertices are $\Xi$. A convex tiling of $S^2$ by $\frac{2k}{k-2}$-gons determines an element $(M,\Xi)$ by declaring $\Xi_p$ to be the lattice of differences of vertices (in the case hexagon-tilings, differences of black vertices).

We first define the lattice $\Lambda$ based on a number of choices. Choose a point $(M,\Xi)$ in the open stratum and a singularity $p_{2k}$. For $j=1,\dots,2k-1$, let $\gamma_j$ be the shortest path connecting $p_{2k}$ to $p_j$, making a choice if necessary. Relabel the singularities so that as $j$ increases, the $\gamma_j$ are cyclically ordered about $p_{2k}$, in clockwise order. The $\gamma_j$ are straight lines in the flat structure which intersect only at the endpoint $p_{2k}$. Finally, choose an isometric trivialization of $\Xi$ to $\Z[\zeta_k]\subset \C$ over the complement of $\bigcup_j\gamma_j$. When $k=3$, we require this trivialization to send the tiling of the tangent space to the standard tiling of $\C$, see Figure \ref{bihex}. Having made these choices, there is, up to translation, a unique isometric immersion $$M-\textstyle\bigcup_j\gamma_j\rightarrow \C$$ compatible with the trivialization of $\Xi$. By Proposition 7.1 of \cite{thurston}, this immersion is an embedding whose image is the interior of a polygon $\mathcal{P}_M$.

Each path $\gamma_j$ corresponds to a pair of adjacent edges $v_j$, $w_j$ of $\mathcal{P}_M$.  Gluing $v_j$ to $w_j$ by an oriented isometry of $\C$ reproduces $(M,\Xi)$, see Figure \ref{square}. Each directed edge $v_j$ or $w_j$ can be viewed as an element of $\C$. The following equations are satisfied: \begin{enumerate}
\item[(1)] $w_j=-\zeta_kv_j$ and \vspace{2pt}
\item[(2)] $\sum_j(v_j+w_j)=0.$
\end{enumerate}

\begin{figure}
\includegraphics[width=2.5in]{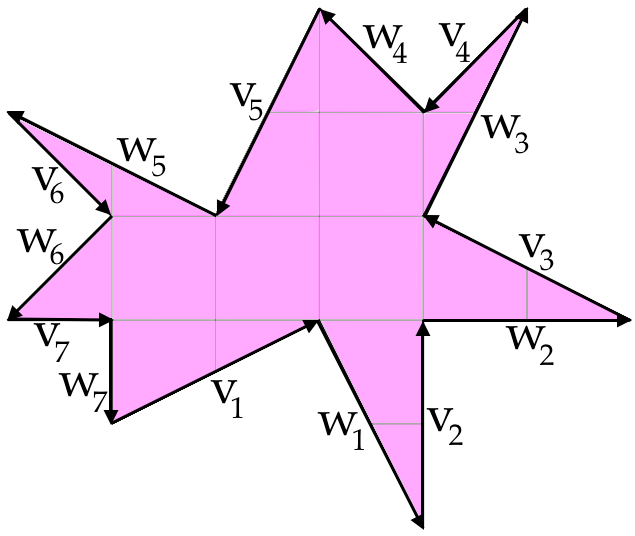}
\caption{A flat cone sphere admitting a square-tiling with cone angle deficits $2\pi/4$ in which $v_i$ and $w_i$ are identified by a rotation.}
\label{square}
\end{figure}

\begin{figure}
\includegraphics[width=1.9in]{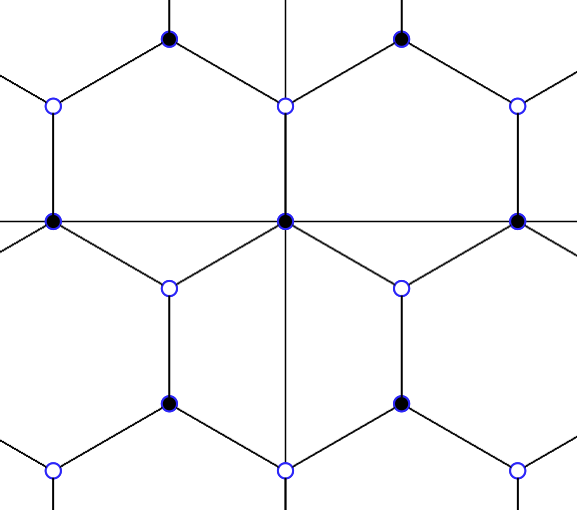}
\caption{The standard tiling of $\C$ by bicolored hexagons with black vertices in $\Z[\zeta_3]$.}
\label{bihex}
\end{figure}

Equation (1) holds because the cone angle deficit at $p_j$ is $\alpha_j$, hence $v_j$ and $w_j$ meet at a vertex of $\mathcal{P}_M$ with exterior angle $\alpha_j$. Equation (2) holds because the vectors $v_j$ and $w_j$ put end-to-end must close up to form $\mathcal{P}_M$. By (1) and (2), the vectors $v_1,\dots,v_{2k-2}$ determine $\mathcal{P}_M$ uniquely. A neighborhood of $$(v_1,\dots,v_{2k-2})\in \C^{2k-2}$$ is an orbifold chart on $\mathcal{M}_{\frac{2\pi}{k}, \dots,\frac{2\pi}{k}}^{rig}$. A small deformation of $(v_1,\dots,v_{2k-2})$ defines a small deformation of the polygon $\mathcal{P}_M$, which in turn defines a small deformation of $(M,\Xi)$.

The area of $\mathcal{P}_M$ is a quadratic form in $(v_1,\dots,v_{2k-2})$. By Proposition 3.3 of \cite{thurston}, the area extends naturally to a Hermitian form $A$ whose signature is $(1,2k-3)$. We define $$\Lambda:=\Z[\zeta_{k}]^{2k-2}\subset (\C^{2k-2},A).$$ In the local coordinate chart on $\mathcal{M}_{\frac{2\pi}{k}, \dots,\frac{2\pi}{k}}^{rig}$ defined by $(v_1,\dots,v_{2k-2})$, elements of $\Lambda$ correspond to tilings of the flat cone sphere by intersecting $\mathcal{P}_M$ with the unique translate of the planar tiling by $\frac{2k}{k-2}$-gons whose vertices include the vertices of $\mathcal{P}_M$. This correspondence holds even in the closure of the chart, as $\mathcal{P}_M$ degenerates. These closed charts cover $\overline{\mathcal{M}}^{rig}_{\frac{2\pi}{k},\dots,\frac{2\pi}{k}}$. 

It remains to prove that the construction of $\Lambda$ is independent of choice of coordinate chart. Choosing the data of a labelling of the singularities, geodesics $\gamma_j$, and a trivialization of $\Xi$, we have produced an identification \begin{align*} H^1(M,\Xi\otimes \R) & \cong \C^{2k-2} \\ H^1(M,\Xi) & \cong \Lambda. \end{align*} Another choice of coordinate chart will therefore preserve $\Lambda$. The area of $(M,\Xi)$ is the same regardless of the chart, so this automorphism also preserves the Hermitian form $A$. Alternatively, we may note that $A$ is a rescaling of the cup product on $H^1(M,\Xi)$ induced by the Hermitian pairing on $\Xi$ with values in $\Z[\zeta_k]$. Therefore the transition function between two coordinate charts lies in $U(\Lambda)$. Furthermore, since the cup product on $H^1(M,\Xi)$ is valued in $\Z[\zeta_k]$, there is a rescaling of $A$ which induces on $\Lambda$ the structure of a lattice over $\Z[\zeta_k]$.

Let $\Gamma\subset U(\Lambda)$ be the monodromy group of the principal stratum. Then there is a map $$D\,:\,\overline{\mathcal{M}}_{\frac{2\pi}{k}, \dots,\frac{2\pi}{k}}^{rig}\rightarrow \Gamma\backslash\mathcal{C}^+$$ where $\mathcal{C}^+\subset \C^{1,2k-3}$ is the positive cone. 
 It follows from Theorem 4.1 of \cite{thurston} that the induced map on the projectivization $$\mathbb{P}D\,:\,\mathbb{P}\overline{\mathcal{M}}_{\frac{2\pi}{k}, \dots,\frac{2\pi}{k}}^{rig}=\overline{\mathcal{M}}_{\frac{2\pi}{k}, \dots,\frac{2\pi}{k}}\rightarrow \mathbb{P}\Gamma\backslash \mathbb{CH}^{2k-3}$$ is an isomorphism of complex-hyperbolic orbifolds. The fibers of the projectivization map are isomorphic on both sides to $\C^*/\langle \zeta_k\rangle$. Since $D$ is an isomorphism on an individual fiber and $\mathbb{P}D$ is an isomorphism, so is $D$. We have shown the map $D$ identifies convex tilings of the sphere with $\Gamma\backslash \Lambda^+$.\end{proof}

\begin{proposition}\label{weight1} Let $v\in \Gamma\backslash \Lambda^+$ correspond to a convex tiling $\mathcal{T}$ of the sphere with cone angle deficits $\{2\pi \mu_i/k\}$. Let $\textrm{Aut}^+(\mathcal{T})$ denote the group of oriented isomorphisms of the tiling $\mathcal{T}$. Then $$|\textrm{Stab}_\Gamma(v)|=|\textrm{Aut}^+(\mathcal{T})|\prod_i \frac{\mu_i!}{(1-\mu_i/k)^{\mu_i-1}}.$$ \end{proposition}

\begin{proof} By Proposition \ref{stuff}, the tiling $\mathcal{T}$ defines a point in $\overline{\mathcal{M}}^{rig}_{\frac{2\pi}{k}, \dots,\frac{2\pi}{k}}\cong \Gamma\backslash \mathcal{C}^+$. Therefore, the stabilizer $\textrm{Stab}_\Gamma(v)$ is isomorphic to the local orbifold group at this point. The projectivization map $$\overline{\mathcal{M}}^{rig}_{\frac{2\pi}{k}, \dots,\frac{2\pi}{k}}\rightarrow \overline{\mathcal{M}}_{\frac{2\pi}{k}, \dots,\frac{2\pi}{k}}$$ defines an isomorphism of local orbifold groups at any point, because the kernel of $\Gamma\rightarrow \mathbb{P}\Gamma$ acts freely on $\mathcal{C}^+$. For each cone point of $\mathcal{T}$ with cone angle deficit $2\pi \mu_i/k$, the local orbifold group at $[\mathcal{T}]\in \overline{\mathcal{M}}_{\frac{2\pi}{k}, \dots,\frac{2\pi}{k}}$ contains a normal subgroup $\Gamma_i$ of order $$|\Gamma_i|=\frac{\mu_i!}{(1-\mu_i/k)^{\mu_i-1}}$$ coming from braiding the $\mu_i$ singularities which have collided. Moreover, the discussion after Theorem 4.1 of \cite{thurston} shows that there is an exact sequence $$0\rightarrow \prod_i\Gamma_i\rightarrow \textrm{Stab}_\Gamma(v)\rightarrow \textrm{Aut}^+(\mathcal{T})\rightarrow 0.$$ \end{proof}

\begin{definition}\label{weight} The {\it weight} of a tiling $\mathcal{T}$ whose deficits form the partition $\mu=(\mu_1,\dots,\mu_n)$ of $2k$ is $$wt(\mathcal{T})=\frac{1}{|\textrm{Aut}^+(\mathcal{T})|}\prod_{i=1}^n \frac{(1-\mu_i/k)^{\mu_i-1}}{\mu_i!}=\frac{1}{|\textrm{Stab}_\Gamma(v)|}.$$ \end{definition} Note that the generic weight---when none of the singularities have collided and $\mathcal{T}$ has no nontrivial automorphisms---is one.

\section{The Hermitian Form on $\Lambda$}

\begin{figure}
\includegraphics[width=2.7in]{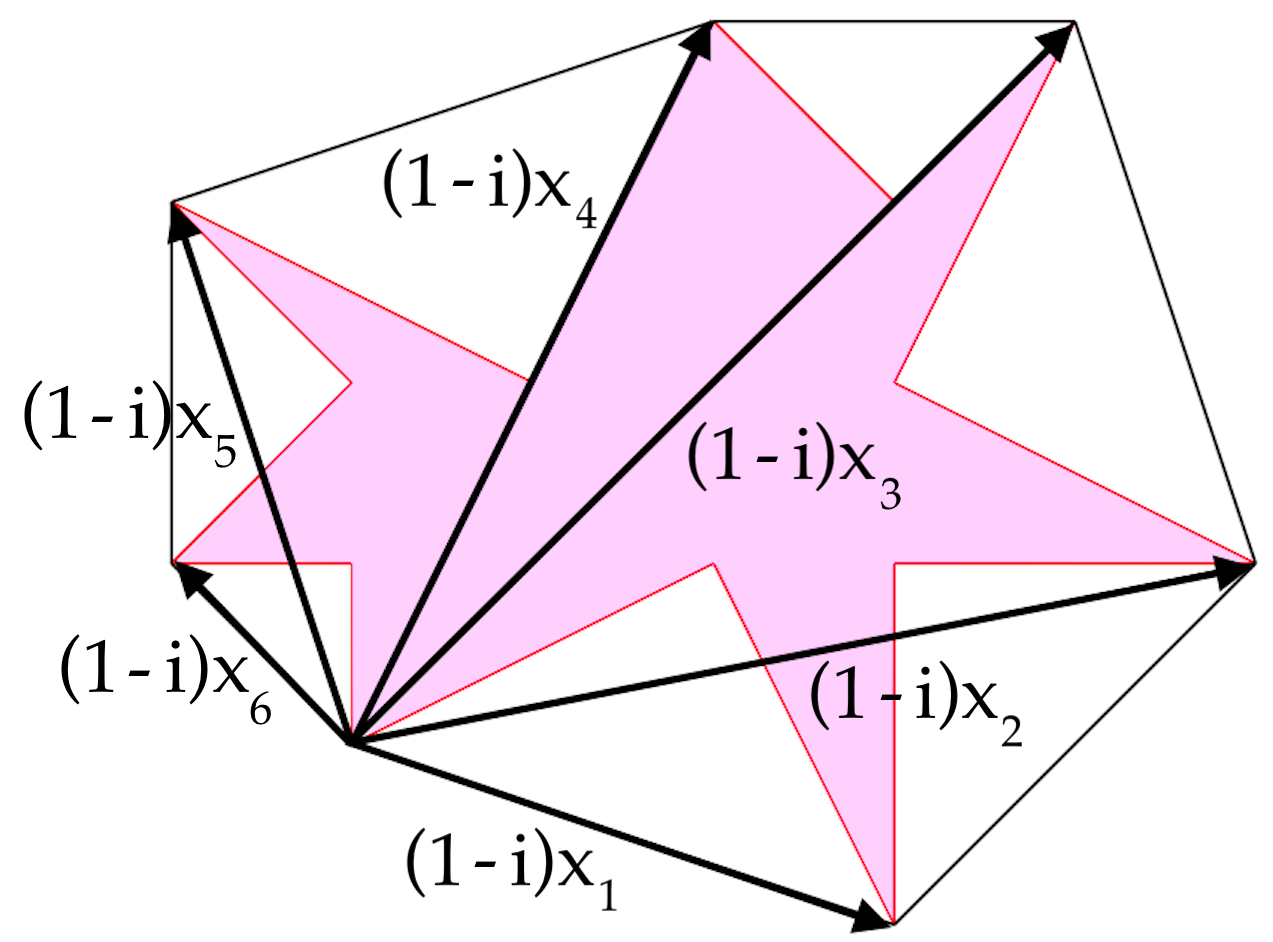}
\caption{The vectors $(1-i)x_i$. The shaded region is $\mathcal{P}_M$.}
\label{area}
\end{figure}

We now determine the Hermitian form $A$ explicitly by writing the area of $M$ as a function of $(v_1, \dots, v_{2k-2})$. The Gram matrix is simpler if we change the basis of $\Lambda$ by working in the coordinates $x_i = \sum_{j \leq i} v_j$ for $i=1,\dots,2k-2$. Equation (1) in the proof of Proposition \ref{stuff} implies $$(1 - \zeta_k)x_i = \sum_{j \leq i} v_j + w_j.$$ The area of $M$ is the area of the convex hull of $\mathcal{P}_M$ minus the sum of the areas of the triangles with edges $v_i$ and $w_i$, cf. Section 7 of \cite{thurston}. See Figure \ref{area}. Since a triangle with two edges $z_1,z_2\in\C$ has area $\frac{1}{2}\textrm{Im}(\overline{z}_1 z_2)$, the area of the convex hull is
\[
\sum_{i=1}^{2k-3}\frac{1}{2}\mathrm{Im}\left(\overline{(1- \zeta_k)x_i}\cdot(1-\zeta_k)x_{i+1}\right) = \sum_{i=1}^{2k-3}\frac{|1-\zeta_k|^2}{2}\mathrm{Im}(\overline{x}_ix_{i+1})
\]
Setting $x_0 = 0$, we have $v_i = x_i - x_{i-1}$ for all $i$. Using again $w_j = -\zeta_k v_j$, the area of the triangle whose edges are $v_i$ and $w_i$ is
\[
-\frac{1}{2}\mathrm{Im}(\overline{v}_iw_i) = \frac{1}{2}\mathrm{Im}((\overline{x_i - x_{i-1}})\cdot\zeta_k(x_i - x_{i-1})).
\]
Therefore the area of $M$ is $\frac{1}{2}\mathrm{Im}(\sum_{i,j} \overline{x}_i B_{ij} x_j)$ where the matrix $B$ is
\[
\begin{pmatrix} - 2 \zeta_k & \zeta_k & 0 & \dots & 0 \\ \zeta_k & - 2 \zeta_k & \zeta_k  & \dots & 0 \\  0 & \zeta_k & -2\zeta_k   & \dots & 0 \\ \vdots & \vdots & \vdots & & \vdots \\ 0 & 0 & 0 & \dots & -2\zeta_k \end{pmatrix} + \begin{pmatrix} 0 & |1-\zeta_k|^2  & 0 & \dots & 0 \\ 0 & 0 & |1-\zeta_k|^2 & \dots & 0 \\  0 & 0 & 0  & \dots & 0 \\ \vdots & \vdots & \vdots & & \vdots \\ 0 & 0 & 0 & \dots & 0 \end{pmatrix}
\]

\begin{proposition}\label{intersection} Let $\epsilon_k:=(1+\zeta_k)^{-1}.$ The Gram matrix of $A$ in the coordinates $(x_1,\dots,x_{2k-2})$ is $$\textrm{Im}\,\zeta_k\begin{pmatrix} -1 & \epsilon_k & 0 & \dots & 0 \\ \overline{\epsilon}_k & -1 & \epsilon_k  & \dots & 0 \\  0 & \overline{\epsilon}_k & -1   & \dots & 0 \\ \vdots & \vdots & \vdots & & \vdots \\ 0 & 0 & 0 & \dots & -1 \end{pmatrix}.$$ \end{proposition}

\begin{proof} 
The Hermitian inner product is given by $A=\frac{B - B^\dagger}{4i}$. The proposition follows from direct computation.
\end{proof} 

\begin{definition}\label{scalings} Define two Hermitian inner products $\star$ and $\bullet$ on $\Lambda$ by \begin{align*} x\star y&:= \frac{|1+\zeta_k|^2}{\textrm{Im}\,\zeta_k}A(x, y) \\ x\bullet y&:=\frac{2}{|1+\zeta_k|^2}(x\star y)\end{align*} Define the real inner product $\cdot$ on $\Lambda$ by $$x\cdot y:=\textrm{Re}\,(x\bullet y).$$ \end{definition}

By Proposition \ref{intersection}, $\star$ endows $\Lambda$ with the structure of a Hermitian lattice over $\Z[\zeta_k]$. The real inner product $\cdot$ endows $\Lambda$ with the structure of an even $\Z$-lattice---by comparing with the area of a fundamental tile, $x \cdot x$ is the number of triangles, and twice the number of squares or hexagons. The Hermitian product $\bullet$ is only used to aid in computations in Section 4.

%

\begin{proposition}\label{dual} Let $\Lambda^\vee$ be the Hermitian dual of $\Lambda$ with respect to $\star$. Then $\Lambda=(1+\zeta_k)\Lambda^\vee.$ \end{proposition}

\begin{proof} Let $G$ be the Gram matrix of $(\Lambda,\star)$. By Proposition \ref{intersection} and Definition \ref{scalings}, every entry of $G$ is divisible by $1+\zeta_k$. Therefore $\Lambda\subset (1+\zeta_k)\Lambda^\vee$.  An inductive argument shows that $\det(G) = -|1+\zeta_k|^{2k-2}$. The index of $\Lambda$ in $\Lambda^\vee$ is the square norm of $\det(G)$ and $(1+\zeta_k)\Lambda^\vee$ already has index $|1+\zeta_k|^{4k-4}$ in $\Lambda^\vee$. Thus we conclude $\Lambda=(1+\zeta_k)\Lambda^\vee$. \end{proof}

\begin{definition} A Hermitian lattice $\Lambda$ over $\Z[\zeta_k]$ is {\it $\alpha$-modular} if $\Lambda=\alpha\Lambda^\vee$. \end{definition}

\begin{proposition}\label{dual2} Let $\Lambda$ be a $(1+\zeta_k)$-modular lattice over $\Z[\zeta_k]$. Define $x\cdot y :=\frac{2}{|1+\zeta_k|^2}\textrm{Re}(x\star y)$ as in Definition \ref{scalings} and let $\Lambda^*$ denote the dual of $\Lambda$ with respect to $\cdot$. Then $\Lambda=(1-\zeta_k)\Lambda^*.$ In particular, $(\Lambda,\cdot)$ is an even unimodular lattice when $k=6$.\end{proposition}

\begin{proof} We compute \begin{align*} x\cdot \frac{y}{1-\zeta_k} =\frac{2}{|1+\zeta_k|^2}\textrm{Re}\left(x \star \frac{y}{1-\zeta_k}\right) \in \,&\frac{2}{|1+\zeta_k|^2}\textrm{Re}\left(\frac{1+\zeta_k}{1-\zeta_k}\Z[\zeta_k]\right)=\Z.\end{align*} Therefore $\Lambda\subset (1-\zeta_k)\Lambda^*$. To show the reverse containment, we compute the index of $(\Lambda,\cdot)$ in its dual. The co-volume of $(\Lambda,\cdot)$ is \begin{align*} \textrm{CoVol}(\Lambda,\cdot) & =\left(\frac{2}{|1+\zeta_k|^2}\right)^{2k-2}\textrm{CoVol}(\Lambda,\star)  \\ &=\left(\frac{2}{|1+\zeta_k|^2}\right)^{2k-2} |\det(G)|\, \textrm{CoVol}(\Z[\zeta_k]\subset \C)^{2k-2} \\ & = \left(\frac{2}{|1+\zeta_k|^2}\right)^{2k-2}\cdot|1+\zeta_k|^{2k-2}\cdot (\textrm{Im}\,\zeta_k)^{2k-2} \\ &=|1-\zeta_k|^{2k-2}.\end{align*}

The index of a lattice in its dual is the square of its co-volume. The proposition follows, as $(1-\zeta_k)\Lambda^*$ already has index $|1-\zeta_k|^{4k-4}$ in $\Lambda^*$. \end{proof}
%

\begin{proposition}\label{unique} Let $k=6$, $4$, or $3$. There is at most one $(1+\zeta_k)$-modular lattice over $\Z[\zeta_k]$ of indefinite signature $(r,s)$. \end{proposition}

\begin{proof} When $k=3$, the lattice is unimodular, so the result follows from Theorem 7.1 of \cite{allcock2}. An argument of Basak, cf. Lemma 2.6 of \cite{basak}, deduces the $k=6$ case. Similarly, the $k=4$ case follows from Basak's argument, and the uniqueness of odd, indefinite unimodular lattices over $\Z[i]$. \end{proof}

When $k=6$, Allcock \cite{allcock} also identified $(\Lambda,\star)$ as the unique $(1+\zeta_6)$-modular lattice of signature $(1,9)$.

\begin{remark}\label{index1} The Hermitian form for $(\Lambda,\star)$ bears a similarity to the Dynkin diagram for $A_{2k-2}$. Let $(\beta_1,\dots,\beta_{2k-2})$ be the basis of $\Lambda$ associated to the coordinates $(x_1,\dots,x_{2k-2})$. Each $\beta_i$ defines a complex reflection of $\Lambda$: $$r_{\beta}(\lambda)=\lambda-(1+\zeta_k)\frac{\beta\star \lambda}{\beta\star \beta}\beta .$$ The braid group of the singularities admits a representation into the monodromy group $\Gamma$ of the rigidified moduli space, and the braid switching $p_j$ with $p_{j+1}$ maps to the complex reflection $\beta_j$.\end{remark}

When $k=6$, Laza \cite{laza} verified that $\Gamma=U(\Lambda)$ is the Hermitian isometry group of $\Lambda$, using that the monodromy group is generated by the reflections $\beta_j$.  When $k=3$, Allcock \cite{allcock2} sketched a proof that the reflection group $\Gamma\subset U(\Lambda)$ is an index $2$ subgroup whose center is $\langle \zeta_3\rangle$. In the following proposition, we prove that $\Gamma=U(\Lambda)$ when $k=4$, and note that our proof method generalizes to the other two cases. 

\begin{proposition} When $k=6$ or $4$, $\Gamma=U(\Lambda)$. When $k=3$, $\Gamma$ has index $2$ in $U(\Lambda)$. \end{proposition}

\begin{proof} Allcock and Laza have proven the proposition for $k=6$ and $k=3$. Suppose $k=4$. There is only one convex tiling of the sphere by $1$ square, given by the doubling of a 45-45-90 triangle. Therefore, there is only one $\Gamma$-orbit of vectors of norm $2$ in $(\Lambda,\star)$. Hence, there is also only one $U(\Lambda)$-orbit of norm $2$ vectors. To prove $\Gamma=U(\Lambda)$ it suffices to show that $\textrm{Stab}_\Gamma(v)$ and $\textrm{Stab}_{U(\Lambda)}(v)$ are equal for some $v$ satisfying $v\star v=2$. By Proposition \ref{weight1}, $|\textrm{Stab}_\Gamma(v)|=2^{13}3^2.$ On the other hand, $|\textrm{Stab}_{U(\Lambda)}(v)|=|U(v^\perp)|.$ Let $L$ be the negative-definite $(1+i)$-modular lattice $$L:=\{(a,b)\in \Z[i]^2\,\big{|}\,u\equiv v\textrm{ mod }1+i\}$$ where $\Z[i]^2$ has the diagonal Hermitian form $\textrm{diag}(-1,-1)$. Let $H$ be the unique $(1+i)$-modular lattice of signature $(1,1)$, whose Gram matrix is $$\twobytwo{0}{1-i}{1+i}{0}.$$ By Proposition \ref{unique} we have $\Lambda\cong L\oplus L \oplus H.$ The orthogonal complement of a vector $v$ of norm $2$ is $$v^\perp\cong L\oplus L\oplus \langle -2\rangle.$$ The automorphisms of $v^\perp$ act by a fourth root of unity on the vector generating $\langle -2\rangle$---this vector must be preserved up to units since it is perpendicular to all other vectors of norm $-2$. In addition, an automorphism can either preserve or switch the two factors of $L$ and acts by an element of $U(L)$ on each factor. Thus, $$|U(v^\perp)|=2^3\cdot |U(L)|^2.$$ To compute the order of $U(L)$, note that $U(L)$ acts transitively that all $24$ vectors of norm $-2$ in $L$. Then the stabilizer of say $(1+i,0)\in L$ has order $4$, given by multiplying the second coordinate by a fourth root of unity. Hence $|U(L)|=24\cdot 4=2^5\cdot 3$. So $$|\textrm{Stab}_{U(\Lambda)}(v)|=2^3\cdot (2^5\cdot 3)^2=2^{13}\cdot 3^2.$$ We conclude that $\Gamma=U(\Lambda)$. \end{proof}

\section{Theta Series of Hermitian Lattices}

In this section, let $\Lambda$ denote a $(1+\zeta_k)$-modular Hermitian lattice over $\Z[\zeta_k]$ of indefinite signature $(r,s)$, which by Proposition \ref{unique} is unique if it exists. A point $p\in Gr^+(r,\Lambda\otimes \C)$ in the complex Grassmannian corresponds to a positive definite subspace $V^+\subset \Lambda \otimes \C$. Let $V^-$ denote the negative definite orthogonal complement. Given any $v\in \Lambda$, let $v^+\in V^+$ and $v^-\in V^-$ be its projections. Our starting point is the Siegel theta function \cite{siegel} on $Gr^+(r,\Lambda\otimes \C)\times \mathbb{H}$: $$\Theta(p,\tau)=\sum_{v\in \Lambda} q^{\frac{1}{2}v^+\cdot v^+}\overline{q}^{\,-\frac{1}{2}v^-\cdot v^-}=\sum_{v\in \Lambda} q^{\frac{1}{2}v\cdot v}|q|^{-v^-\cdot v^{-}}$$ where $q=e^{2\pi i \tau}$. More generally, the theta function $\Theta(p,\tau)$ can be defined in the same way for any $\Z$-lattice, where $p$ ranges over the positive real Grassmannian.

\begin{definition} A function $f(\tau)$ on the upper half-plane is {\it modular} of weight $(r,s)$ for a subgroup $\Gamma\subset SL_2(\Z)$ if $$f(\gamma\cdot \tau)=(c\tau+d)^r(c\overline{\tau}+d)^sf(\tau)$$ for all $\gamma=\twobytwo{a}{b}{c}{d}\in \Gamma$. \end{definition}

By Proposition \ref{dual}, the $\Z$-lattice $(\Lambda,\cdot)$ is abstractly isometric to $u\Lambda^*$ for $u=|1-\zeta_k|$. This implies a nice modularity property of the Siegel theta function:

\begin{proposition}\label{umod} Suppose $(\Lambda,\cdot)$ is an even $\Z$-lattice of signature $(2r,2s)$ such that $\Lambda$ is abstractly isometric to $u\Lambda^*$ with $u>0$. Define $$\Gamma_u:= \left\langle \twobytwo{1}{1}{0}{1},\twobytwo{1}{0}{u^2}{1}\right\rangle\subset SL_2(\Z).$$ Then for fixed $p$, the theta function $\Theta(p, \tau)$ is modular of weight $(r,s)$ for the group $\Gamma_u$. \end{proposition}

\begin{proof} Note that the integrality of $\cdot$ implies that $u^2\in \Z$. First observe $\Theta(p,\tau)$ is invariant under $\tau\mapsto \tau+1$ because $\Lambda$ is even. The Poisson summation formula implies that $$\Theta\left(p,\frac{-1}{u^2\tau}\right)=i^{s-r}(u\tau)^r(u\overline{\tau})^s\Theta(p,\tau).$$ 
Thus $\Theta(p,\tau)$ is almost modular with respect to $\tau\mapsto -1/u^2\tau$, but for the factor $i^{s-r}$. Since $$\twobytwo{1}{0}{u^2}{1}=\twobytwo{0}{-1}{u^2}{0}\twobytwo{1}{1}{0}{1}^{-1}\twobytwo{0}{-1}{u^2}{0}^{-1},$$ the powers of $i$ cancel in the transformation rule for the matrix $\twobytwo{1}{0}{u^2}{1}$. \end{proof}

\begin{corollary}\label{umod2} Suppose $k=6$, $4$, or $3$. Let $\Lambda$ be a $(1+\zeta_k)$-modular Hermitian lattice over $\Z[\zeta_k]$ of signature $(r,s)$. For fixed $p$, the Siegel theta function $\Theta(p,\tau)$ of $\Lambda$ is modular of weight $(r,s)$ for the group $$\Gamma_1(|1-\zeta_k|^2):=\left\{\gamma\in SL_2(\Z)\,\big{|}\,\gamma\equiv \twobytwo{1}{*}{0}{1}\textrm{ mod }|1-\zeta_k|^2\right\}.$$ \end{corollary}

\begin{proof} When $k=6$, $4$, or $3$, we have $|1-\zeta_k|^2=1$, $2$, or $3$ respectively. In these cases, SAGE \cite{sage} verifies that $\Gamma_u=\Gamma_1(|1-\zeta_k|^2)$. This is a peculiarity of the small value of $u^2$. \end{proof}

For fixed $\tau$, the function $\Theta(p,\tau)$ is absolutely invariant with respect to the action of $U(\Lambda)$ on $Gr^+(r,\Lambda\otimes \C)$, that is, the action on the variable $p$. It defines a {\it theta correspondence}---one can integrate against a function $f(\tau)$ on $\mathbb{H}$ to produce a function on $U(\Lambda)\backslash Gr^+(r,L\otimes \C)$ or vice versa. Much research has focused on former process, called the ``Borcherds lift," which has led to beautiful product formulas for coefficients of modular forms \cite{borcherds}. In this paper, we integrate against the $p$ variable.

\begin{definition}  Suppose $k=6$, $4$, or $3$. Let $(\Lambda,\star)$ be the unique Hermitian lattice of signature $(1,s)$ over $\Z[\zeta_k]$ such that $\Lambda=(1+\zeta_k)\Lambda^\vee$. Suppose $s>1$. Define $$g_{\Lambda}(\tau):=\frac{(\textrm{Im}\,\tau)^s}{|Z(U(\Lambda))|}\,\int_{U(\Lambda) \backslash \mathbb{CH}^s} \Theta(p,\tau)\,dp$$ where $dp$ is the complex hyperbolic volume form and $Z(U(\Lambda))$ is the center. \end{definition}

Note that the center $Z(U(\Lambda))$ is isomorphic to the group of units in $\Z[\zeta_k]$. Satz 1 of Siegel's foundational paper \cite{siegel} on theta functions of indefinite Hermitian forms proves that $g_\Lambda(\tau)$ is a Maass form, which we now define: 

\begin{definition}\label{maass} A {\it Maass form} of weight $w$ for $\Gamma\subset SL_2(\Z)$ is a real-analytic function $f(\tau)$ on $\mathbb{H}$ satisfying $f(\gamma\cdot \tau)=(c\tau+d)^wf(\tau)$ for all $\gamma\in \Gamma$
which has a Fourier expansion of the form $$c^{-}(0)(\textrm{Im}\,\tau)^w+\sum_{n<0} c^-(n)\Gamma(s,4\pi |n|\,\textrm{Im}\,\tau)q^n+ \sum_{n\geq 0} c^+(n)q^n$$ and polynomial growth as $\tau$ approaches a rational cusp of $\mathbb{H}$. Here $$\displaystyle\Gamma(t,z):=\int_z^\infty x^{t-1}e^{-x}\,dx$$ is the incomplete gamma function. \end{definition}

See Section 7 of \cite{ono} for a general introduction to Maass forms.

\begin{remark} Define the {\it weight $w$ hyperbolic Laplacian} to be $$\Delta_w=-(\textrm{Im}\,\tau)^2\partial_\tau\partial_{\overline{\tau}}+iw\textrm{Im}\,\tau\, \partial_{\overline{\tau}}.$$ Then a Maass form $f(\tau)$ of weight $w$ is harmonic with respect to $\Delta_w$, that is $\Delta_wf=0$. This condition plus polynomial growth at the cusps implies the existence of a Fourier expansion as above. \end{remark}



In the following lemmas, we explicitly compute the Fourier coefficients of $g_\Lambda(\tau)$, and in Theorem \ref{integrals} reprove Siegel's theorem. The positive Fourier coefficients are of particular relevance to the enumeration of tilings, because they are essentially the weighted number of $U(\Lambda)$-orbits of lattice points.

We first collect the terms of $\Theta(p,\tau)$ into $U(\Lambda)$-orbits: $$\Theta(p,\tau)=\sum_{[v]\in U(\Lambda)\backslash \Lambda} \left(\,\sum_{w\in [v]} |q|^{-w^-\cdot w^{-}}\right) q^{\frac{1}{2}v\cdot v}.$$ Since $q^{\frac{1}{2}v\cdot v}$ is independent of $p$, we have \begin{align}\label{one} g_{\Lambda}(\tau)=\frac{(\textrm{Im}\,\tau)^s}{|Z(U(\Lambda))|}\sum_{[v]\in U(\Lambda)\backslash \Lambda} \left(\int_{U(\Lambda)\backslash \C\mathbb{H}^s} \sum_{w\in[v]} |q|^{-w^-\cdot w^-}\,dp\right)q^{\frac{1}{2}v\cdot v}.\end{align}
It is useful to introduce the following function on $U(\Lambda)\backslash \Lambda\times \mathbb{H}$:
\begin{align}\label{two} F([v],\tau)&:=\frac{1}{|Z(U(\Lambda))|}\int_{U(\Lambda)\backslash \C\mathbb{H}^s} \sum_{w\in[v]} |q|^{-w^-\cdot w^-}\,dp. \end{align} For $v\neq 0$ we have \begin{align*} F([v],\tau)&=\int_{\textrm{Stab}(v)\backslash\C\mathbb{H}^s} |q|^{-v^-\cdot v^-}\,dp.\end{align*}

The factor of $|Z(U(\Lambda))|^{-1}$ disappears (except when $v=0$) because the scalar matrices in $U(\Lambda)$ act  trivially on the fundamental domain $U(\Lambda)\backslash\mathbb{CH}^s$. There are four possibilities for the behavior of this integral, depending on when $v^2=0$, $v^2>0$, $v^2<0$, or $v$ is isotropic. In the next four lemmas, we compute the integral $F([v],\tau)$ in these four cases.

\begin{lemma}\label{zero} Suppose $v=0$. Then $$F([v],\tau)=\frac{{\rm Vol}(U(\Lambda)\backslash \mathbb{CH}^s)}{|Z(U(\Lambda))|}.$$ \end{lemma}

\begin{proof} The lemma follows immediately from (2). \end{proof}

\begin{lemma}\label{positive} Suppose $v\cdot v=2n>0$. Then $$F([v],\tau)=\frac{(n\textrm{Im}\,\tau)^{-s}}{|{\rm Stab}(v)|}.$$ \end{lemma}

\begin{proof} Since $\textrm{Stab}(v)$ is a finite group, we can rewrite \begin{align*}F([v],\tau)&=\frac{1}{|\textrm{Stab}(v)|}\int_{\mathbb{CH}^s}|q|^{-v^-\cdot v^-}\,dp \\ &=\frac{1}{|\textrm{Stab}(v)|}\int_{\C\mathbb{H}^s} e^{4\pi n\textrm{Im}\,\tau (e_0^{-}\cdot e_0^{-})}\,dp
.\end{align*} where $e_0=v/\sqrt{2n}$. Extend $e_0$ to an orthonormal basis of $(\C^{1,s},\bullet)$, where $\bullet$ is the Hermitian form whose real part is $\cdot$ as in Definition \ref{scalings}. In this basis, the Klein model of $\mathbb{CH}^s$ is the unit ball in the plane $z_0=1$, i.e. the points $p=(1,z_1,\dots,z_k)$ with $\sum |z_i|^2<1$. Let $r=\sqrt{|z_1|^2+\dots+|z_k|^2}$ denote the radial coordinate of this unit ball. The volume form of the homogeneous metric whose curvature is pinched between $1/4$ and $1$ is $$dp=\frac{4^sr^{2s-1}\,dr\,d\mu}{(1-r^2)^{s+1}}$$ where $d\mu$ is the standard volume form on $S^{2s-1}$. To express the projections $e_0^+$ and $e_0^-$ in terms of $p$, we need to be careful to use the Hermitian inner product $\bullet$ since we are projecting to the {\em complex} span of $p$ and its orthogonal complement; with this in mind, $e_0^+=\frac{e_0\bullet p}{p\bullet p}p$ and so $$e_0^-\cdot e_0^-=e_0\cdot e_0-e_0^+\cdot e_0^+=1 - \frac{|e_0\bullet p|^2}{p\bullet p}=-\frac{r^2}{1-r^2}.$$ Let $C:=4\pi n\textrm{Im}\,\tau$. Then, we have \begin{align*}
F([v],\tau)&=\frac{4^s\textrm{Vol}(S^{2s-1})}{|\textrm{Stab}(v)|}\int_{0}^1e^{-C\frac{r^2}{1-r^2}}\frac{r^{2s-1}}{(1-r^2)^{s+1}}\,dr \\
 &=\frac{4^s\textrm{Vol}(S^{2s-1})}{2C^s|\textrm{Stab}(v)|}\int_{0}^{\infty}e^{-u}u^{s-1}\,du \hspace{20pt}\textrm{ where } u=C\frac{r^2}{1-r^2} \\
 &= \frac{4^s(s-1)!\textrm{Vol}(S^{2s-1})}{2C^s|\textrm{Stab}(v)|} \\ &=\frac{(n\textrm{Im}\,\tau)^{-s}}{|\textrm{Stab}(v)|}. 
\end{align*}\end{proof}

\begin{lemma}\label{negative} Suppose that $v\cdot v=-2n<0$. Then $$F([v],\tau)=(n\textrm{Im}\,\tau)^{-s}\Gamma(s,4\pi n\,\textrm{Im}\,\tau)\frac{{\rm Vol}(X_v)}{(4\pi)^{s-1}}$$ where $X_v$ is the quotient of $\mathbb{CH}^{s-1}$ by $\textrm{Stab}(v)$. \end{lemma}

\begin{proof} We have $$F([v],\tau)=\int_{\textrm{Stab}(v)\backslash \mathbb{CH}^s} e^{4\pi n\,\textrm{Im}\tau(e_1^{-}\cdot e_1^{-})}dp$$ where $e_1=v/\sqrt{2n}$. As in Lemma \ref{positive}, extend $e_1$ to an orthonormal basis of $(\C^{1,s},\bullet)$. Let $p=(1,z_1,\dots,z_s)$ be an element of $\mathbb{CH}^s$ and let $z'=(z_2,\dots,z_s)$. Then $$e_1^{-}\cdot e_1^{-}=-1-\frac{|e_1\bullet p|^2}{p\bullet p}=-\frac{1-|z'|^2}{1-|z_1|^2-|z'|^2}.$$ We may re-write $F([v],\tau)$ as $$\int_{\textrm{Stab}(v)\backslash \mathbb{CH}^s} \textrm{exp}\left(-4\pi n\,\textrm{Im}\tau\frac{1-|z'|^2}{1-|z_1|^2-|z'|^2}\right) \frac{4^s\,dV_{2s}}{(1-|z_1|^2-|z'|^2)^{s+1}}$$ where $dV_{2s}$ is the Euclidean volume form on the unit ball in $z_0=1$. 

Since $\textrm{Stab}(v)$ preserves the $z_1$ coordinate, it acts on the totally geodesic copy of $\mathbb{CH}^{s-1}$ on which $z_1 = 0$. Let $X_v$ be a fundamental domain for this action. Then the set of $p=(1,z_1,z')$ such that $z'\in X_v$ is a fundamental domain for $\textrm{Stab}(v)\backslash \mathbb{CH}^s$. To simplify the notation let $C:=4\pi  n\,\textrm{Im}\tau$ and $h(z') := \sqrt{1-|z'|^2}$. Then \begin{align*} F([v],\tau)&=4^s\int_{X_v} \left(\int_0^{2\pi}\!\!\!\int_0^h e^{-C\frac{h^2}{h^2-r^2}}\frac{r\,dr\,d\theta}{(h^2-r^2)^{s+1}}\right) dV_{2s-2} \\ &=4^s \int_{X_v} \left( \pi C^{-s}h^{-2s} \int_{C}^{\infty} u^{s-1}e^{-u}\,du \right)\,dV_{2s-2} \\ &\hspace{150pt}\textrm{ where } u=C\frac{h^2}{h^2-r^2}\\ &= 4\pi C^{-s}\Gamma(s, C) \int_{X_n} \frac{4^{s-1}dV_{2k-2}}{(1-|z'|)^k} \\ &=(n\textrm{Im}\,\tau)^{-s}\Gamma(s,4\pi n\,\textrm{Im}\,\tau)\frac{\textrm{Vol}(X_v)}{(4\pi)^{s-1}} \end{align*} \end{proof}

Since $s$ is a positive integer, the incomplete gamma function has the simple form $\Gamma(s,C) = (s-1)!e^{-C} \sum_{j=0}^{s-1}\frac{C^j}{j!}$. 

\begin{lemma}\label{null} Suppose $v=\beta e$ for some $e$ primitive isotropic and nonzero $\beta\in\Z[\zeta_k]$. Then $$F([v],\tau)= \frac{(2\,\textrm{Im}\,\zeta_k)^{2s-1}(s-1)!}{(2\pi\beta\overline{\beta}\textrm{Im}\,\tau)^s\,|{\rm Aut}(e^\perp/e)|}.$$ \end{lemma}

\begin{proof} Since $\Lambda=(1+\zeta_k)\Lambda^\vee$, an easy argument shows that there is an isotropic vector $f$ such that $f\star e=1+\zeta_k$. So $H:=\Z[\zeta_k]e\oplus \Z[\zeta_k]f$ generates a copy of the unique $(1+\zeta_k)$-modular lattice of signature $(1,1)$. Then $H^{\perp}\cong e^\perp/e$ is a negative definite $(1+\zeta_k)$-modular lattice. Elements of $\textrm{Stab}(v)\subset U(\Lambda)$ correspond to triples $(T,w_0,m)$ where $T$ is a Hermitian isometry of $H^{\perp}$, $w_0\in H^\perp$, and $m\in \Z$ through the action
\begin{align*} e&\mapsto e \\ f&\mapsto f+\left(-\frac{w_0\star w_0}{|1+\zeta_k|^2}+(1-\overline{\zeta}_k)m\right) e+w_0 \\ w&\mapsto -(1+\zeta_k)^{-1}(w_0\star T w)e+T w \hspace{20pt}\textrm{for }w\in H^\perp.\end{align*} 

Let $\{e_2,\dots,e_s\}$ be an orthonormal basis of $H^\perp\otimes\C$. Define $$\tilde{e}:=\frac{1+\overline{\zeta}_k}{2}e$$ so that $f\bullet \tilde{e}=1$.
Let $(x,y,z_2,\dots,z_s)$ be coordinates for $\mathbb{C}^{1,s}$ in the basis $\{\tilde{e},f,e_2,\dots,e_k\}$ and let $z'=(z_2,\dots,z_s)$. Setting $y=1$, the tuples $(x, z')$ such that $2\textrm{Re}(x) - |z'|^2 > 0$ give coordinates on $\mathbb{CH}^s$. Since $v^+=\frac{v\bullet p}{p\bullet p}p$ and $v \bullet p = \frac{2}{|1+\zeta_k|^2} v \star p = \frac{2\overline{\beta}}{1+\zeta_k}$, we have $$v^-\cdot v^-=-v^+\cdot v^+=\frac{-4\beta\overline{\beta}}{|1+\zeta_k|^2(2\,\textrm{Re}(x)-|z'|^2)}.$$ 

We wish to compute the integral $F([\beta e_0],\tau)$:
$$\int_{\textrm{Stab}(v)\backslash\mathbb{CH}^s} \textrm{exp}\left(- \frac{8\pi\beta\overline{\beta}\,\textrm{Im}\,\tau}{|1+\zeta_k|^2(2\,\textrm{Re}(x)-|z'|^2)}\right)\,\frac{4^s\,dV_{2s}}{(2\,\textrm{Re}(x)-|z'|^2)^{s+1}}.$$ Let $a=\frac{1}{2\,\textrm{Re}(x)-|z'|^2}$ and $b=\textrm{Im}(x)$. We have $$\frac{-2 \,dV_{2s}}{(2\,\textrm{Re}(x)-|z'|^2)^2}=da\,db\,dV_{2s-2}.$$ Let $\Gamma_0\subset \textrm{Stab}(v)$ be the finite index subgroup such that $T=Id$. We may enlarge the domain of integration to $\Gamma_0\backslash \mathbb{CH}^s$ and divide by the index $[\textrm{Stab}(v):\Gamma_0]=|\textrm{Aut}(H^\perp)|$. The subgroup $\Gamma_0$ acts by translations in $H^\perp$ on the $z'$ coordinate and by translations by $\frac{4\textrm{Im}\,\zeta_k}{|1+\zeta_k|^2}$ on the $b$ coordinate. \begin{align*} F([\beta e_0],\tau)&= \frac{4^s}{ 2|\textrm{Aut}(H^\perp) | } \int_{H^\perp\backslash \C^{s-1}} \int_0^{\frac{4\textrm{Im}\,\zeta_k}{|1+\zeta_k|^2}}\int_0^\infty \!\!e^{-\frac{8\pi\beta\overline{\beta}\textrm{Im}\,\tau}{|1+\zeta_k|^2} a}a^{s-1}\,da\,db\,dV_{2s-2} \\ &=\frac{(2\,\textrm{Im}\,\zeta_k) 4^s(s-1)! \left(\frac{8}{|1+\zeta_k|^2}\pi\beta\overline{\beta}\,\textrm{Im}\,\tau\right)^{-s}\textrm{CoVol}(H^\perp)}{ |1+\zeta_k|^2 |\textrm{Aut}(H^\perp )| } \\ &= \left(\frac{2\pi\beta\overline{\beta}\,\textrm{Im}\,\tau}{(2\textrm{Im}\,\zeta_k)^2}\right)^{-s} \frac{(s-1)!}{2\,\textrm{Im}\,\zeta_k\,|\textrm{Aut}(H^\perp)|}.\end{align*} Note that CoVol$(H^\perp)=|1-\zeta_k|^{2s-2}$ because the lattice $(H^\perp,\cdot)$ satisfies $H^\perp=(1-\zeta_k)(H^\perp)^*$. The lemma follows because $H^\perp\cong e^\perp/e$.\end{proof}

\begin{proposition}\label{orbits} The orbits of primitive isotropic vectors in $\Lambda$ are in bijection with negative definite $(1+\zeta_k)$-modular lattices of rank $s-1$. \end{proposition}

\begin{proof} The bijection sends the orbit of $e$ to the lattice $e^\perp/e$. Every negative definite $(1+\zeta_k)$-modular lattice $L$ of rank $s-1$ embeds into $\Lambda$ because $\Lambda\cong L\oplus H$ by Proposition \ref{unique}. Hence, there exists for all $L$ a primitive isotropic vector $e\in \Lambda$ such that $e^\perp/e\cong L$.

On the other hand, if $e_1^\perp/e_1 \cong e_2^\perp/e_2$, choose as in the proof of Lemma \ref{null} primitive isotropic $f_1$ and $f_2$ such that $f_i\star e_i=1+\zeta_k$. Then there is an isometry of $\Lambda$ which sends $e_1\mapsto e_2$, $f_1\mapsto f_2$, and $$\{e_1,f_1\}^\perp=e_1^\perp/e\rightarrow e_2^\perp/e_2=\{e_2,f_2\}^\perp.$$ \end{proof}

\begin{definition} Define the {\it mass} to be $m_{s-1}:=\sum_L |\textrm{Aut}(L)|^{-1}$ where $L$ ranges over all isomorphism types of $(1+\zeta_k)$-modular lattices of rank $s-1$. \end{definition}

We now combine the previous four lemmas with Proposition \ref{orbits}:

\begin{theorem}\label{integrals}The function $g_{\Lambda}(\tau)$ is a Maass form of weight $1-s$ for $\Gamma_1(|1-\zeta_k|^2)$, whose Fourier coefficients are: \begin{align*}c^{-}(0) &=  \displaystyle\frac{{\rm Vol}(U(\Lambda)\backslash\mathbb{CH}^s)}{|Z(U(\Lambda))|} (\textrm{Im}\,\tau)^s \\[0.2em]
c^{-}(n) &= \displaystyle  n^{-s}\!\!\!\sum_{\substack{[v]\in U(\Lambda)\backslash\Lambda \\ v\cdot v=-2n}}\frac{{\rm Vol}(X_v)}{(4\pi)^{s-1}} \\[0.2em]
c^{+}(0) &= \displaystyle\frac{(2\,\textrm{Im}\,\zeta_k)^{2s-1}(s-1)!}{(2\pi)^s} \zeta_{\Z[\zeta_k]}(s) \,m_{s-1} \\[0.2em]
c^{+}(n) &= \displaystyle\sum_{\substack{[v]\in U(\Lambda)\backslash\Lambda \\ v\cdot v=2n}}\frac{n^{-s}}{|\textrm{Stab}(v)|}
\end{align*}
where $\zeta_{\Z[\zeta_k]}(s)=\sum_{(\beta)\neq (0)} (\beta\overline{\beta})^{-s}$ is the Dedekind zeta function. \end{theorem}

\begin{proof} From (\ref{one}) and (\ref{two}), we have $$g_\Lambda(\tau)=(\textrm{Im}\,\tau)^s\sum_{[v]\in \Gamma\backslash \Lambda} F([v],\tau)q^{\frac{v\cdot v}{2}}.$$ We have computed in Lemmas and \ref{zero}, \ref{positive}, \ref{negative}, and \ref{null} the value of $F([v],\tau)$ for all $[v]\in U(\Lambda)\backslash\Lambda$. Grouping the terms for $v=0$, $v\cdot v=-2n$, and $v\cdot v=2n$ respectively give the above formulas for $c^-(0)$, $c^-(n)$, and $c^+(n)$. Grouping the terms for $v$ isotropic gives $$c^+(0)=\sum_{ \substack{ [\beta e]\in U(\Lambda)\backslash \Lambda \\ e\textrm{ primitive} \\ \textrm{isotropic} }}  \frac{(2\,\textrm{Im}\,\zeta_k)^{2s-1}(s-1)!}{(2\pi\beta\overline{\beta})^s\,|\textrm{Aut}(e^\perp/e)|}.$$ Summing over $\beta$ gives the term $\zeta_{\Z[\zeta_k]}(s)$. Then, by Proposition \ref{orbits}, the sum over orbits of primitive isotropic vectors $e$ gives the mass term $m_{s-1}$.

We have shown that $g_\Lambda(\tau)$ has the appropriate Fourier series expansion for a Maass form. By Corollary \ref{umod2}, $\Theta(p,\tau)$ is modular of weight $(1,s)$ for $\Gamma_1(|1-\zeta_k|^2)$. Since $\textrm{Im}\,\tau$ is of weight $(-1,-1)$ for any $\gamma\in SL_2(\Z)$, we conclude that $g_{\Lambda}$ is of weight $(1-s,0)$. That is, $$g_{\Lambda}(\gamma\cdot \tau)=(c\tau +d)^{1-s}g_{\Lambda}(\tau)$$ for all $\gamma\in \Gamma_1(|1-\zeta_k|^2)$.

Finally, we sketch a proof that $g_{\Lambda}(\tau)$ has polynomial growth at the cusps of $\mathbb{H}$. One can show $\Theta(p,\tau)$ is bounded by $C(p)((\textrm{Im}\,\tau)^{-s}+1)$ where $C(p)$ is bounded by an absolute constant times $1+a^{-1}$---here $a$ is the coordinate of $p$ near a cusp as defined in Lemma \ref{null}. By the computations in Lemma \ref{null}, the pushforward of the complex hyperbolic volume form to the $a$ coordinate is proportional to $a^{s-1}da$. Thus the integral of $C(p)\,dp$ over $U(\Lambda)\backslash \mathbb{CH}^s$ converges. We conclude that the integral of $\Theta(p,\tau)\,dp$ is bounded by some constant times $(\textrm{Im}\,\tau)^{-s}+1$. Therefore $g_{\Lambda}(\tau)$ is bounded by a constant times $1+(\textrm{Im}\,\tau)^s$.

Hence $g_{\Lambda}(\tau)$ is a Maass form of weight $1-s$ for $\Gamma_1(|1-\zeta_k|^2)$. \end{proof}

\begin{corollary}\label{main} Let $\Lambda$ be a $(1+\zeta_k)$-modular lattice of signature $(1,s)$ and let $\Gamma$ be a finite index subgroup of $U(\Lambda)$. Then $$[U(\Lambda):\Gamma]\left(\frac{1}{2\pi i}\partial_\tau\right)^sg_{\Lambda}(\tau)=-\frac{s!\textnormal{Vol}(\Gamma\backslash \mathbb{CH}^s)}{|Z(\Gamma)| (4\pi)^s}+ \sum_{\substack{[v]\in \Gamma\backslash \Lambda \\  v \cdot v>0}} \frac{1}{|\textnormal{Stab}_\Gamma(v)|}q^{\frac{1}{2}v\cdot v}$$ is a weight $1+s$ modular form for $\Gamma_1(|1-\zeta_k|^2)$. \end{corollary}

\begin{proof} First, we may as well assume $\Gamma=U(\Lambda)$ because taking a finite index subgroup $\Gamma\subset U(\Lambda)$ multiplies the righthand side by a factor of $$[U(\Lambda):\Gamma]=\frac{|Z(U(\Lambda))|}{|Z(\Gamma)|}[\mathbb{P}U(\Lambda):\mathbb{P}\Gamma].$$

Let $D=\frac{1}{2\pi i}\partial_\tau=q\partial_q$. If $f(\tau)$ is a Maass form of weight $1-s$ with $s>1$ then by Theorem 1.1 of \cite{raising} $$D^sf=(-1)^ss!\frac{c_0}{(4\pi)^s}+ \sum_{n>0} c^+(n)n^sq^n$$ is a holomorphic modular form of weight $1+s$.
Applying this theorem to $g_{\Lambda}(\tau)$ and multiplying by $[U(\Lambda):\Gamma]$ gives the corollary, by the explicit computation of the coefficients of $g_{\Lambda}(\tau)$ in Theorem \ref{integrals}. Note $(-1)^s=-1$ because $s$ must be odd for there to exist a $(1+\zeta_k)$-modular lattice $\Lambda$ of signature $(1,s)$. \end{proof}

\section{Formulas for Numbers of Tilings} In this section, we give formulas for the number of convex triangle-, square-, and hexagon-tilings of the sphere.  Recall that the weight of a tiling $\mathcal{T}$ by $k$-gons with angle deficits $\{2\pi \mu_i/k\}$ is $$wt(\mathcal{T})=\frac{1}{|\textrm{Aut}^+(\mathcal{T})|}\prod_{i=1}^n \frac{(1-\mu_i/k)^{\mu_i-1}}{\mu_i!}.$$

\begin{theorem}\label{formula} The weighted number of oriented convex triangulations of $S^2$ with $2n$ triangles is $$\sum_{|\mathcal{T}|=2n} wt(\mathcal{T})= \frac{809}{2^{15}3^{13}5^2}\sigma_9(n).$$ The weighted number of oriented convex square-tilings with $n$ squares is $$\sum_{|\mathcal{T}|=n}wt(\mathcal{T})=\frac{1}{2^{13}3^2}(\sigma_5(n)+8\sigma_5(n/2))$$ where $\sigma_m(n)=0$ when $n\notin \Z$. The weighted number of oriented, convex hexagon-tilings with $n$ hexagons (see Convention \ref{convention}) is $$\sum_{|\mathcal{T}|=n}wt(\mathcal{T})= \frac{1}{2^33^4}(\sigma_3(n)-9\sigma_3(n/3)).$$ \end{theorem}

\begin{remark} Using the group of all automorphisms Aut$(\mathcal{T})$ in the definition of the weight and counting tilings without orientation halves the above formulas. \end{remark}

\begin{proof}[Proof of Theorem \ref{formula}] By Proposition \ref{stuff}, there is a bijective correspondence between convex tilings $\mathcal{T}$ and orbits $[v]\in\Gamma\backslash \Lambda^+$ where $v \cdot v$ is the number of triangles, or twice the number of squares or hexagons. Furthermore, by Proposition \ref{weight1}, $wt(\mathcal{T})=|\textrm{Stab}_{\Gamma}(v)|^{-1}$. Hence, the generating function $h_k(q)$ for the weighted number of convex tilings by $\frac{2k}{k-2}$-gons is $$h_k(q)=\!\!\sum_{\substack{[v]\in \Gamma\backslash \Lambda^+ \\ v\cdot v>0}} \frac{1}{|\textrm{Stab}_\Gamma(v)|}q^{\frac{1}{2}v\cdot v}.$$ By Proposition \ref{dual}, the lattice $\Lambda$ is the unique $(1+\zeta_k)$-modular lattice of signature $(1,2k-3)$ over $\Z[\zeta_k]$. Thus Corollary \ref{main} implies that \begin{align}\label{volumeeq} \tilde{h}_k(q):=h_k(q)-\frac{(2k-3)!\textrm{Vol}(\Gamma\backslash \mathbb{CH}^{2k-3})}{k(4\pi)^{2k-3}}\end{align} is a modular form of weight $2k-2$ for $\Gamma_1(|1-\zeta_k|^2)$. Note that $|Z(\Gamma)|=k$ in all three cases. We now treat each case.

{\bf Triangles:} We have $\tilde{h}_6(q)$ is of weight $10$ for $\Gamma_1(1)=SL_2(\Z)$. The space of weight 10 modular forms for $SL_2(\Z)$ is one-dimensional over $\C$ and spanned by the weight $10$ Eisenstein series $E_{10}(\tau)=-\frac{1}{264}+\sum_{n>0} \sigma_9(n)q^n.$ Hence there is a constant $A$ such that $$\tilde{h}_6(q)=AE_{10}(\tau).$$ There are two triangulations consisting of two triangles, depicted on the left in Figure \ref{fig1}. The corresponding partitions $\mu$ of $12$ are $(2,5,5)$ and $(4,4,4)$ and the oriented automorphism groups have orders $2$ and $6$, respectively. Thus $$A=\frac{1}{2\cdot (3)(155520)^2}+\frac{1}{6\cdot (648)^3}=\frac{809}{2^{15}\cdot 3^{13}\cdot 5^2}.$$

{\bf Squares:} We have that $\tilde{h}_4(q)$ is of weight $6$ for $\Gamma_1(2)$. The space of weight $6$ modular forms for $\Gamma_1(2)$ is two-dimensional over $\C$ and spanned by the weight $6$ Eisenstein series $E_6(\tau)= -\frac{1}{504}+\sum_{n>0} \sigma_5(n)q^n$ and $E_6(2\tau)$. Let $B$ and $C$ be constants such that $$\tilde{h}_4(q)=BE_6(\tau)+CE_6(2\tau).$$ There is one square-tiling consisting one square, whose partition is $(2,3,3)$ and whose oriented automorphism group has order $2$, on the left in Figure \ref{fig2}. There are three square-tilings with two squares, with partitions $(2,3,3)$, $(1,1,3,3)$, and $(2,2,2,2)$ and automorphism groups of orders $2$, $2$, and $8$, respectively, on the right in Figure \ref{fig2}. Solving the resulting system of equations, we have \begin{align*} B&=\frac{1}{2^{13}\cdot 3^2} \\ C&=\frac{1}{2^{10}\cdot 3^2}. \end{align*}

{\bf Hexagons:} We have that $\tilde{h}_3(q)$ is of weight $4$ for $\Gamma_1(3)$. The space of weight $4$ modular forms for $\Gamma_1(3)$ is two-dimensional over $\C$ and spanned by the weight $4$ Eisenstein series $E_4(\tau)=\frac{1}{240}+\sum_{n>0}\sigma_3(n)q^n$ and $E_4(3\tau)$. Hence there are constants $D$ and $F$ such that $$\tilde{h}_3(\tau)=DE_4(\tau)+FE_4(3\tau).$$ There is one hexagon-tiling consisting of one hexagon satisfying Convention \ref{convention}, with curvatures $(2,2,2)$ and automorphism group of order $3$, on the left in Figure \ref{fig3}. There is one hexagon-tilings with two hexagons, with curvatures $(2,2,1,1)$ automorphism groups of order 2, on the right in Figure \ref{fig3}. Therefore \begin{align*} D&=\frac{1}{2^3\cdot 3^4} \\ F&=-\frac{1}{2^3\cdot 3^2}. \end{align*} \end{proof}

\begin{example} We verify the formula in Theorem \ref{formula} for triangulations with four triangles. The righthand side of Figure \ref{fig1} depicts the non-negative curvature triangulations with four triangles. The partitions are $(3,4,5)$, $(3,3,3,3)$, $(1,1,5,5)$, and $(2,2,4,4)$ and the oriented automorphism groups have orders $1$, $12$, $2$, and $4$ respectively. Then the formula correctly states that \begin{align*} \frac{1}{1\cdot (24)(648)(155520)}+ \frac{1}{12\cdot (24)^4}&+\frac{1}{2\cdot (155520)^2}+\frac{1}{4\cdot (3)^2(648)^2} \\ =\frac{809}{2^{15}3^{13}5^2}&(1+2^9).\end{align*} \end{example}

Finally, we remark that Equation (\ref{volumeeq}) in Theorem \ref{formula} and the determination of the constants $A$, $B$, $C$, $D$, and $F$ allows us to compute the complex hyperbolic volumes of the moduli spaces $\Gamma\backslash\mathbb{CH}^{2k-3}$. When $k=6$, we have $$\textrm{Vol}(\Gamma\backslash \mathbb{CH}^9)=\frac{6(4\pi)^9}{9!}\cdot \frac{1}{264}\cdot \frac{809}{2^{15}3^{13}5^2}=\frac{809\pi^9}{2^6\cdot 3^{17}\cdot 5^3\cdot 7\cdot 11}.$$ This is exactly the volume computed in \cite{mcmullen}, which serves as a check for the formula in Theorem \ref{formula}. Similarly, when $k=4$ and $s=5$, we have $$\textrm{Vol}(\Gamma\backslash\mathbb{CH}^5)=\frac{4(4\pi)^5}{5!}\cdot \frac{1}{504}\cdot 9\cdot \frac{1}{2^{13}3^2}=\frac{\pi^5}{2^7\cdot 3^3\cdot 5\cdot 7},$$ which also agrees with the result in \cite{mcmullen}. Finally, when $k=3$ and $s=3$, Theorem \ref{formula} again correctly predicts $$\textrm{Vol}(\Gamma\backslash\mathbb{CH}^3)=\frac{3(4\pi)^3}{3!}\cdot \frac{1}{-240}\cdot (-8)\cdot \frac{1}{2^33^4}=\frac{2\pi^3}{3^5\cdot 5}.$$ 

\begin{figure}[H]
\includegraphics[width=1.3in]{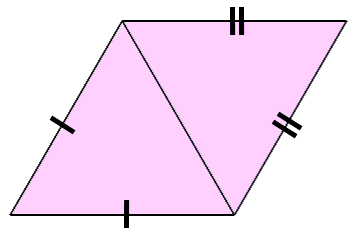}
\includegraphics[width=1.55in]{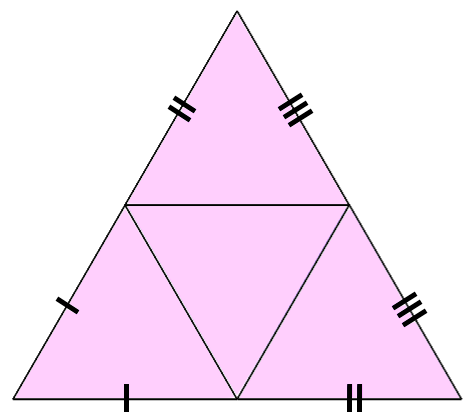}\hspace{10pt}
\includegraphics[width=1.55in]{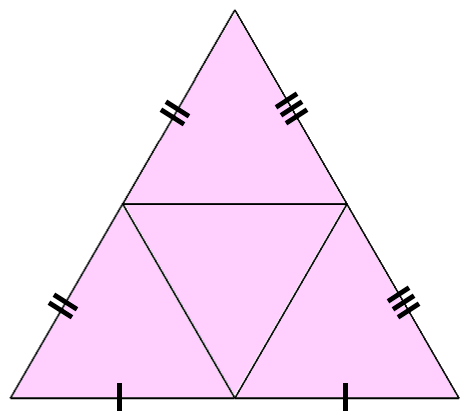}

\includegraphics[width=1.3in]{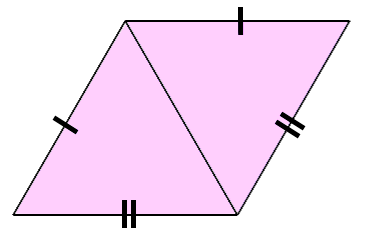}
\includegraphics[width=1.75in]{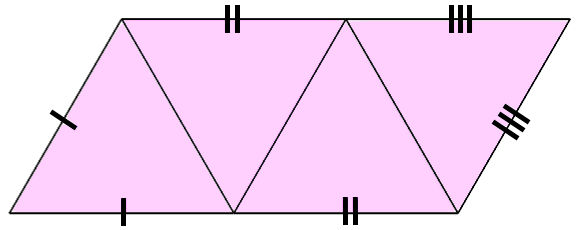}
\includegraphics[width=1.75in]{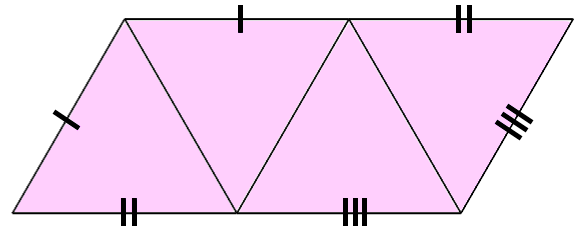}
\caption{The convex triangulations with two or four triangles.}
\label{fig1}
\end{figure}

\begin{figure}[H]
\includegraphics[width=0.95in]{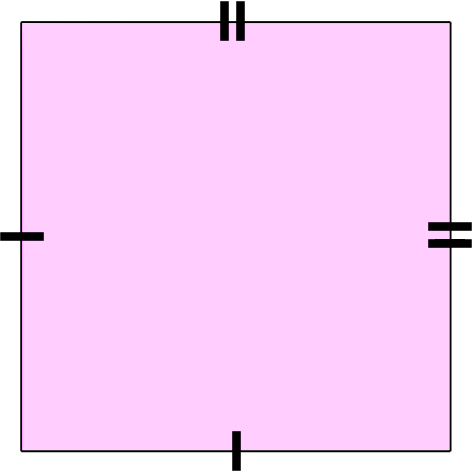}\hspace{10pt}
\includegraphics[width=0.95in]{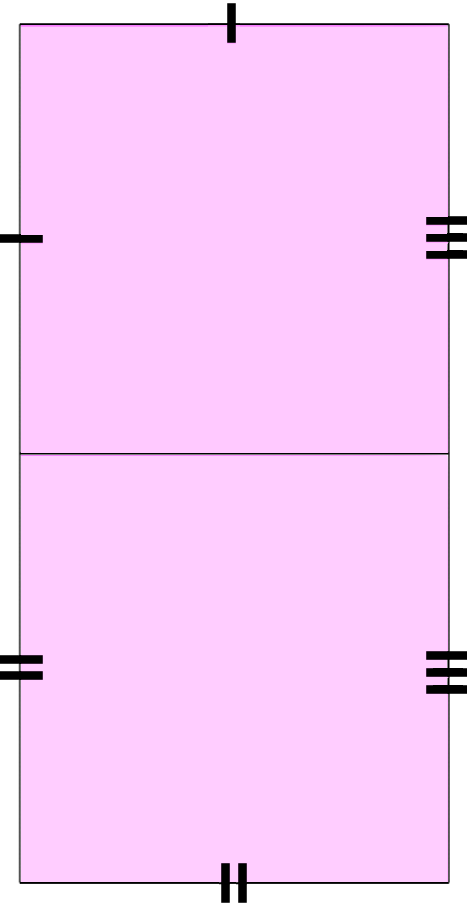}\hspace{10pt}
\includegraphics[width=0.95in]{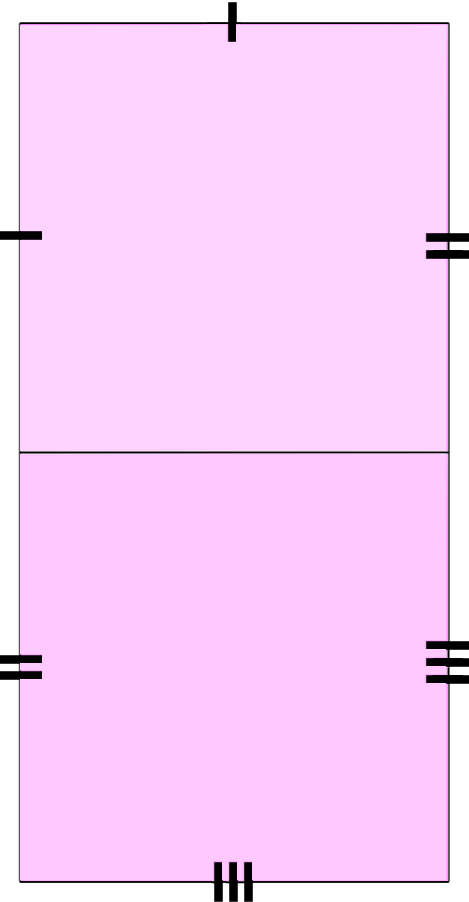}\hspace{10pt}
\includegraphics[width=0.95in]{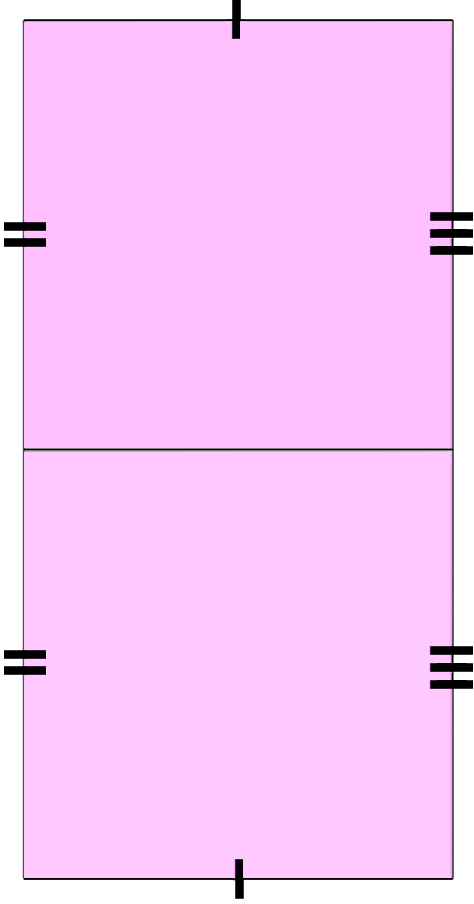}
\caption{The convex square-tilings with one or two squares.}
\label{fig2}
\end{figure}

\begin{figure}[H]
\includegraphics[width=1.35in]{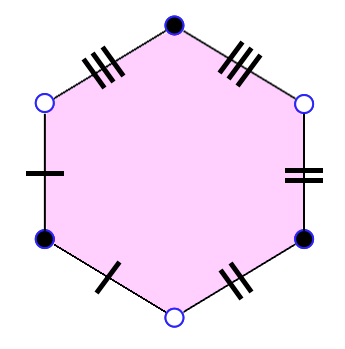}\hspace{10pt}
\includegraphics[width=2.2in]{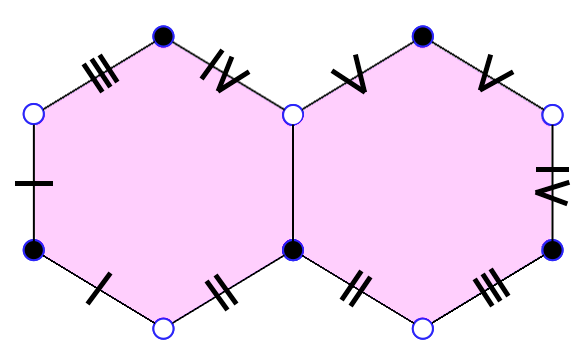}
\caption{The convex hexagon-tilings with one or two hexagons.}
\label{fig3}
\end{figure}


\begin{thebibliography}{99}

\bibitem{alexandrov}
A. D. Alexandrov.
\emph{Convex Polyhedra}.
Springer Science \& Business Media,
2005.

\bibitem{allcock}
D. Allcock.
\emph{The Leech lattice and complex hyperbolic reflections}.
Inventiones mathematicae,
140(2): 283-301,
2000.

\bibitem{allcock2}
D. Allcock.
\emph{New complex- and quaternion-hyperbolic reflection groups}.
Duke Math. Journal 103(2): 303--333, 
2000.

\bibitem{basak}
T. Basak.
\emph{The complex Lorentzian Leech lattice and the bimonster}.
Journal of Algebra, 309(1): 32-56,
2007.

\bibitem{borcherds}
R. E. Borcherds.
\emph{Automorphic forms on $O_{s+ 2, 2}(\R)$ and infinite products.}
Inventiones mathematicae 120.1: 161-213,
1995.

\bibitem{raising}
J. Bruinier, K. Ono, and R. C. Rhoades.
\emph{Differential operators for harmonic weak Maass forms and the vanishing of Hecke eigenvalues}.
Mathematische Annalen, 342(3): 673-693,
2008.

\bibitem{dm}
P. Deligne, and G. D. Mostow.
\emph{Monodromy of hypergeometric functions and non-lattice integral monodromy}.
Publications MathŽmatiques de l'IH\'ES 63.1: 5-89,
1986.




\bibitem{laza}
R. Laza.
\emph{Triangulations of the sphere and degenerations of K3 surfaces}. 
arXiv: 0809.0937,
2008.

\bibitem{looijenga}
E. N. Looijenga.
\emph{Uniformization by Lauricella functions---an overview of the theory of Deligne-Mostow}.
Arithmetic and geometry around hypergeometric functions: 207-244,
Birkh\"auser Basel,
2007.

\bibitem{mcmullen}
C. T. McMullen.
\emph{The Gauss-Bonnet theorem for cone manifolds and volumes of moduli spaces}.
2013.

\bibitem{mostow}
G. D. Mostow.
\emph{Generalized Picard lattices arising from half-integral conditions.}
Publications Math\'ematiques de l'IH\'ES 63: 91-106,
1986.


\bibitem{ono}
K. Ono.
\emph{Unearthing the visions of a master: harmonic Maass forms and number theory}.
Current Developments in Math., Vol. 2008:347-454,
2009.

\bibitem{picard}
\'E. Picard.
\emph {Sur une extension aux fonctions de deux variables du probleme de Riemann relatif aux fonctions hyperg\'eom\'etriques}.
Annales scientifiques de l'\'Ecole Normale SupŽrieure 10: 305-322,
1881. 

\bibitem{sage}
W. A. Stein, et al. Sage Mathematics Software. The Sage Development Team, {\tt http://www.sagemath.org}, 2005.

\bibitem{schwartz}
R. Schwartz. 
\emph{Notes on Shapes of Polyhedra.}
arXiv: 1506.07252,
2015.

\bibitem{siegel}
C. L. Siegel.
\emph{Indefinite quadratische Formen und Funktionentheorie II.}
Mathematische Annalen 124.1: 364-387,
1951.

\bibitem{thurston}
W. P. Thurston.
\emph{Shapes of polyhedra and triangulations of the sphere}.
Geometry and Topology monographs 1: 511-549,
1998.

\bibitem{tutte}
W. T. Tutte.
\emph{A census of planar triangulations}.
Canadian Journal of Math, 14(1): 21-38,
1962.

\end{thebibliography}
\end{document}